\documentclass[12pt,a4paper,twoside]{amsart}
\usepackage[english]{babel}
\usepackage[utf8]{inputenc}
\usepackage{amsmath,amssymb,amsthm,amscd}
\usepackage{graphicx}
\usepackage{fancyhdr}
\usepackage{enumerate}
\usepackage{indentfirst}
\usepackage{geometry}
\usepackage{hyperref}
\usepackage[all]{xy}
\usepackage{calc,color}

\title[Backward bounded solutions]{A Characterization of backward bounded solutions}

\thanks{This research was supported by the National Research Foundation (Grant N0. 2022R1F1A1063007, 2022R1I1A1A01055459), and the second author was supported by Chonnam National University (Grant No. 2023-0912-01). }
\author{Minkyu Kwak, Jihoon Lee and Bataa Lkhagvasuren }
\address{Department of Mathematics, Chonnam National University, Gwangju, Republic of Korea}
\email{mkkwak@jnu.ac.kr,jihoon@chonnam.ac.kr, bataa@chonnam.ac.kr}
\begin{document}

\begin{abstract}
We prove that the collection $\mathcal M_{-\infty}$  of backward bounded solutions for a semilinear evolution equation is  the graph of an upper hemicontinuous set-valued function from the low Fourier modes to the higher Fourier modes, which is invariant and contains the global attractor. We also show that there exists a limit $\mathcal M_{\infty}$ of finite dimensional Lipschitz manifolds $\mathcal M_t$ generated by the time $t$-maps ($t>0$) from the flat manifold $\mathcal M_0$ with the Hausdorff distance and we find $\mathcal M_{\infty} \subset
\mathcal M_{-\infty}$. No spectral gap conditions are assumed.
\end{abstract}

\maketitle

  \par

\maketitle

\numberwithin{equation}{section}
\newtheorem{theorem}{Theorem}[section]
\newtheorem{remark}{Remark}[section]
\newtheorem{definition}{Definition}[section]
\newtheorem{lemma}[theorem]{Lemma}
\newtheorem{corollary}[theorem]{Corollary}
\newtheorem{proposition}[theorem]{Proposition}
\newtheorem{example}[theorem]{Example}

\renewcommand{\footskip}{.5in}

\section{Introduction}
Let us consider a semilinear evolution equation in a Hilbert space $H$ of the form
\begin{equation} \label{1.1}
u_t+Au=F(u),
\end{equation}
where $A$ is a linear closed unbounded positive self-adjoint operator in $H$, and the nonlinearity $F(\cdot)$ is bounded and Lipschitz, i.e., there are positive constants $K_0$ and $K_1$ satisfying
\begin{align}
|F(u)| & \le K_0,\label{1.2} \\
|F(u) - F(v)| &\le K_1|u-v|~\mbox{for~all}~u,v \in H. \notag
\end{align}
We also assume that $F(\cdot)$ has a compact support after truncating (as usual) outside of the absorbing ball (see for example \cite{FST}, \cite{SY}, \cite{T}). Hence there exists $R>0$ such that $F(u) = 0$ for $|u| \ge R$. We also assume that $A^{-1}$ is compact in $H$ and thus there exist eigenvalues of $A$ satisfying
$$
0 < \lambda_1 < \lambda_2 \le \cdots \le \lambda_N < \lambda_{N+1} \le \cdots \to \infty,~\mbox{and}~K_1<\lambda_{N+1}.
$$

Let $P$ denote the spectral projection onto the finite dimensional space spanned by the first $N$-eigenfunctions of $A$, and $Q  =I - P$ the complementary orthogonal projection. Then each $u \in H$ has a unique orthogonal decomposition as the sum of low Fourier modes and higher Fourier modes
$$
u = p+q \in PH \oplus QH.
$$

One of the most important questions is whether the ultimate dynamics of (1.1) is governed by a finite number of Fourier modes and thus one can reduce to a finite dimensional system of ordinary differential equations, which completely describe the long-time dynamics. The inertial manifold theory has been developed for this purpose.
The inertial manifold is a finite dimensional smooth invariant manifold which attracts every orbits at the exponential
rates. Hence this manifold contains the global attractor and when restricted to the inertial manifold, the long-time dynamics can be reduced to a finite system of ODE.

However, the theory requires a spectral gap condition
$$ \lambda_{N+1} - \lambda_N > 2K_1 ,$$
which holds only for special domains and $F(\cdot)$. Even much worse in the case of Navier-Stokes equations.
There were enormous efforts for relaxing these conditions with a few gains (see for example \cite{K}, \cite{KS}, \cite{KZ}, \cite{Z} and the references therein).

As another attempt, we first consider the collection  $\mathcal M_{-\infty}$  of  backward bounded solutions of \eqref{1.1}. The existence of such a solution is proved  by  using the degree theory and Arzela-Ascoli theorem in Section 2.

In Section 3, we consider the collection of finite dimensional Lipschitz manifolds $\mathcal M_t$ generated by the time $t$-map ($t>0$) from the flat manifold $\mathcal M_0$, and construct a limit $\mathcal M_\infty$ of $\mathcal M_t$ with the Hausdorff distance.

In Section 4, we prove that $\mathcal M_{\infty} \subset \mathcal M_{-\infty}$, and $\mathcal M_{-\infty}$ is the graph of an upper hemicontinuous set-valued function from $PH$ to $QH$, which is invariant and contains the global attractor.

\section{Backward bounded solutions}

We will show that a backward bounded solution of \eqref{1.1} is approximated by a sequence of solutions for boundary value problems
\begin{equation} \label{2.1}
\begin{cases}
\dfrac{dp}{dt}+Ap = PF(p+q),~p(0)=p_0, \\
\dfrac{dq}{dt}+Aq = QF(p+q),~q(-n)=0,
\end{cases}
\end{equation}
where $n \in \mathbb N$, $p(t) \in PH$, $q(t) \in QH$ for $t \in [-n,0]$. Equivalently, \eqref{2.1} can be expressed by
\begin{equation} \label{2.2}
\begin{cases}
\dfrac{dp}{dt}+Ap = PF(p+q),~p(n)=p_0, \\
\dfrac{dq}{dt}+Aq = QF(p+q),~q(0)=0.
\end{cases}
\end{equation}

Inspired from Lemma 2.4 in \cite{MS}, the existence of a solution of \eqref{2.2} can be proved by using the degree theory.
\begin{lemma} \label{lemma2.1}
Given $p_0 \in PH$, there exists a solution $(p(t), q(t))$of \eqref{2.2} for $t \in [0, n]$ with $p(n) = p_0$ and $q(0) = 0$.
\end{lemma}
\begin{proof}
For any $n \in \mathbb N$ and $p_{-n} \in PH$, let $(p_n(t), q_n(t))$ be the solution of
$$
\begin{cases}
\dfrac{dp}{dt}+Ap = PF(p+q),~p(0)=p_{-n}, \\
\dfrac{dq}{dt}+Aq = QF(p+q),~q(0)=0,
\end{cases}
$$
where $t \in [0,n]$. For convenience, we use the notations $p_n(t) = p_n(t,p_{-n},0)$ and $q_n(t) = q_n(t,p_{-n},0)$ for $n \in \mathbb N$. Then we complete the proof by showing that for any $p_0 \in PH$ and $n \in \mathbb N$, there exists a point $p_{-n} \in PH$ satisying $p_n(n,p_{-n},0) = p_0$.
Let $\mathcal B_n$ denote the ball
$$
\mathcal B_n = \{ p \in PH: |p| \le 2e^{n\lambda_N}(R+|p_0|) \}.
$$
By the truncation, if $p_{-n} \in \partial \mathcal B_n$ we have
$$
p_n(t,p_{-n},0) = e^{-At}p_n(0)= e^{-At}p_{-n},~\mbox{and}~ q_n(t,p_{-n},0) = 0
$$
for all $t \in [0,n]$, where $\partial \mathcal B_n$ denotes the boundary of $\mathcal B_n$. Consequently we get $p_0 \notin p_n(t, \partial \mathcal B_n, 0)$ for $t \in [0,n]$. For each $t \in [0,n]$, consider a function $f_t:\mathcal B_n \to PH$ defined by
$$
f_t(p) = p_n(t,p,0)~\mbox{for}~p \in \mathcal B_n.
$$
Then the degree of $f_t$ with respect to the point $p_0$ is well-defined, and deg$f_t = ~$deg$f_0$ for all $t \in [0,n]$. Since $p_0 \in \mathcal B_n$ and $f_0$ is the identity map on $\mathcal B_n$, we see that
$$
\mbox{deg}f_t =1~\mbox{for~all}~t \in [0,n].
$$
This implies that $p_0 \in f_n(\mathcal B_n) = p_n(n,\mathcal B_n,0)$, and so there exists a point $p_{-n} \in \mathcal B_n$ satisfying $p_0 = p_n(n,p_{-n},0)$.
\end{proof}

The above lemma implies that for any $p_0 \in PH$ and $n \in \mathbb N$, there exists a sequence $\{ (p_n, q_n):[-n,0] \to H \}_{n \in \mathbb N}$ of solutions of \eqref{2.2} such that
$$
p_n(0) = p_0~\mbox{and}~q_n(-n)=0.
$$
We now apply the Arzela-Ascoli theorem to get a backward bounded solution as we see in the following lemma.

\begin{lemma} \label{lemma2.2}
Let $\mathcal F  = \{ (p_n, q_n):[-n,0] \to H \}_{n \in \mathbb N}$ be a sequence of solutions of \eqref{2.1} such that $p_n(0) = p_0$, $q_n(0) = q_n$ and $q_n(-n) = 0$. Then there exists a subsequence $\{( p_{n_k}, q_{n_k}) \}_{k \in \mathbb N}$ of $\mathcal F$ which converges uniformly to a function, say $(p, q)$, on every compact subset in $(-\infty,0]$. In particular, $u(t) = p(t) + q(t)$ is a backward bounded solution of \eqref{1.1} with $p(0) = \lim_{k \to \infty}p_{n_k}(0)$ and $q(0) = \lim_{k \to \infty}q_{n_k}(0)$.
\end{lemma}

\begin{proof}
From the variation of constants formula, we see that
\begin{align} \label{2.3}
p_n(t) &= e^{-PAt}p_0-\int_t^0 e^{-PA(t-s)}PF(p_n(s)+q_n(s))ds,  \\
q_n(t) &= \int_{-n}^t e^{-QA(t-s)}QF(p_n(s)+q_n(s))ds. \notag
\end{align}
For each $n_0 \in \mathbb N$, let $I_{n_0}= [-n_0, 0]$. Then for any $n \ge n_0$ and $t \in I_{n_0}$, we have
\begin{equation} \label{2.4}
|p_n(t)| \le \big( |p_0| + \frac{K_0}{\lambda_N} \big)e^{-\lambda_Nt},
\end{equation}
\begin{align} \label{2.5}
|q_n(t)| &\le \int_{-n}^t \| e^{-QA(t-s)} \|_{op}| QF |ds  \\
&\le K_0\int_{-\infty}^t \| e^{-QA(t-s)} \|_{op}ds  \notag \\
&= \frac{K_0}{\lambda_{N+1}} \notag
\end{align}
and
\begin{align} \label{2.6}
|A^{1/2}q_n(t)| &\le \int_{-n}^t \|A^{1/2}e^{-QA(t-s)}\|_{op}|QF |ds  \\
&\le K_0 \int_{-\infty}^t \|A^{1/2}e^{-QA(t-s)}\|_{op}ds \notag \\
&= \frac{2K_0}{\sqrt{2}} \frac{1}{\sqrt{\lambda_{N+1}}}. \notag
\end{align}
We note that $\{p_n(t)\}$ is uniformly bounded by \eqref{2.4}. Since $D(A^{1/2})$ is compactly imbedded in $H$, $\{q_n(t)\}$ is precompact in $H$ (for more details, see \cite{SY}).

For any $\varepsilon>0$ and $t \in I_{n_0}$, choose $\delta = \delta(\varepsilon) > 0$ so that
\begin{align*}
|p_n(t) - p_n(t-\delta)| &= \big| \int_{t-\delta}^t \frac{dp_n(s)}{ds} ds \big| \le \int_{t-\delta}^t |\frac{dp_n(s)}{ds}|ds \\
&\le \lambda_Ne^{\lambda N \delta}\delta \big( |p_0| +1 + \frac{K_0}{\lambda_N} + K_0 \delta \big) < \varepsilon
\end{align*}
and
\begin{align*}
|q_n(t) - q_n(t-\delta) &\le \big| (e^{-QA\delta} - 1)q_n( t - \delta) \big| + \big| \int_{t-\delta}^t e^{-QA(t-s)}QFds \big| \\
&\le \frac{2\sqrt{2}K_0\sqrt{\delta}}{e\sqrt{\lambda_{N+1}}}+K_0\delta < \varepsilon.
\end{align*}
Hence $\{(p_n(t), q_n(t))\}$ is equicontinuous on $I_{n_0}$. By the Arzela-Ascoli theorem, there exist a subsequence $\{ (p_{n_k}(t), q_{n_k}(t)) \}_{k \in \mathbb N}$ which converges uniformly to a function  $(p(t), q(t))$, i.e.,
$$
p_{n_k}(t) \to p(t) ~\mbox{and}~q_{n_k}(t) \to q(t)~\mbox{for}~t \in I_{n_0}.
$$
By increasing $n_0 \to \infty$ and applying the diagonal argument, we can deduce that
$$
p_{n_k}(t) \to p(t) ~\mbox{and}~q_{n_k}(t) \to q(t)~\mbox{for}~t \in (-\infty, 0].
$$
Note that $p(0) = \lim_{k \to \infty}p_{n_k}(0)$ and $q(0) = \lim_{k \to \infty}q_{n_k}(0)$.
\end{proof}

Let us denote the collection of backward bounded solutions by
\begin{align*}
\mathcal M_{-\infty} = \{ &(p_0, q_0) \in H: u(t) = p(t)+q(t) ~\mbox{is~a~solution~of}~ \eqref{1.1}~ \mbox{such~ that}\\
&~ p(0) = p_0,~q(0) = q_0,~\mbox{and}~q(t)~\mbox{is~bounded~on~} (-\infty, 0]\}.
\end{align*}
Clearly it is invariant and contains the global attractor.

\section{The Graph Transform Method}

The classical graph transform method due to Hadamard is a standard technique used in finite dimensional dynamical systems to construct locally invariant manifolds, such as the center manifold and the center unstable manifold.
Starting from the flat manifold $\mathcal M_0 = PH$, one then lets the dynamics of the given evolution equation act on $M_0$, thereby obtaining a Lipschitz manifold
\begin{align*}
\mathcal M_t = \{ &u(t) \in H: u(t) = p(t) + q(t)~\mbox{is~a~solution~of}~\eqref{1.1} \\
&\mbox{with}~p(0) = p_0 \in PH,~q(0) = 0 \},
\end{align*}
at each time $t > 0$. Note that $\mathcal M_t$ is Lipschitz in $p_0$ from the Lipschitz continuity of solution with respect to the initial data. We are interested in a limit of $\mathcal M_t$ as $t \to \infty$. Not much is known in an infinite dimensional setting except some special case (e.g., see \cite{MS}).

In this section, we reveal that $\{Q\mathcal M_n\}_{n \in \mathbb N}$ is a Cauchy sequence with the Hausdorff distance $d_H$. To show this, we first prove the following lemma which is crucial throughout the paper.

\begin{lemma} \label{lemma3.1}
Let $u(t)$, $v(t)$ be two solutions of \eqref{1.1},  and let $\rho(t) = Pu(t) - Pv(t)$ and $\sigma(t) = Qu(t) - Qv(t)$. Then for any $0 \le t_0 \le t$, we have
\begin{equation} \label{3.1}
|\sigma(t)| \le K_2 |\rho(t_0)|e^{(K_1-\lambda_1)(t-t_0)} + K_3|\sigma(t_0)|e^{-(\lambda_{N+1}-K_1-\alpha)(t-t_0)}~\mbox{and}
\end{equation}
\begin{equation} \label{3.2}
|\rho(t)| \le |\rho(t_0)|(1+K_4(t-t_0))e^{(K_1-\lambda_1)(t-t_0)}+K_5|\sigma(t_0)|e^{(K_1-\lambda_1)(t-t_0)},
\end{equation}
where $K_2$, $K_3$, $K_4$, $K_5$ are positive  constants depending on $K_1$, $\lambda_1$ and $\lambda_{N+1}$.
\end{lemma}

\begin{proof}
We first note that
\begin{align*}
\frac{d|\rho|}{dt}+\lambda_1|\rho| &\le K_1|\rho| + K_1|\sigma|, \\
\frac{d|\sigma|}{dt}+\lambda_{N+1}|\sigma| &\le K_1|\rho| + K_1|\sigma|,
\end{align*}
and
\begin{align*}
\frac{d}{dt}\big( e^{(\lambda_1-K_1)t}|\rho| \big) &\le K_1e^{(\lambda_1 - K_1)t}|\sigma|, \\
\frac{d}{dt}\big( e^{(\lambda_{N+1}-K_1)t}|\sigma| \big) &\le K_1e^{(\lambda_{N+1} - K_1)t}|\rho|.
\end{align*}
Integrating these inequailities from $t_0$ to $t$, we find that
\begin{align}
 e^{(\lambda_1-K_1)t}|\rho(t)| &\le e^{(\lambda_1-K_1)t_0}|\rho(t_0)| + K_1\int_{t_0}^te^{(\lambda_1-K_1)s}|\sigma(s)|ds,~\mbox{and} \label{3.3} \\
 e^{(\lambda_{N+1}-K_1)t}|\sigma(t)| &\le e^{(\lambda_{N+1}-K_1)t_0}|\sigma(t_0)| + K_1\int_{t_0}^te^{(\lambda_{N+1}-K_1)s}|\rho(s)|ds. \label{3.4}
\end{align}
The last integral of \eqref{3.4} is bounded by
\begin{align*}
K_1&\int_{t_0}^t e^{(\lambda_{N+1}-\lambda_1)\tau}\big\{  e^{(\lambda_1-K_1)t_0}|\rho(t_0)| +K_1\int_{t_0}^{\tau}e^{(\lambda_1 - K_1)s}|\sigma(s)|ds\big\}d\tau \\
&= K_1e^{(\lambda_1 - K_1)t_0}|\rho(t_0)|\int_{t_0}^t e^{(\lambda_{N+1}-\lambda_1)\tau}d\tau + K_1^2\int_{t_0}^t e^{(\lambda_{N+1}-\lambda_1)\tau}\int_{t_0}^{\tau}e^{(\lambda_1-K_1)s}|\sigma(s)|dsd\tau \\
&:= I + II
\end{align*}
The term $I$ is bounded by
$$
\frac{K_1}{\lambda_{N+1}-\lambda_1}e^{(\lambda_1 - K_1)t_0}|\rho(t_0)|e^{(\lambda_{N+1} - \lambda_1)t} .
$$
If we apply the Fubini theorem, the term $II$ becomes
$$
\frac{K_1^2}{\lambda_{N+1}-K_1} \int_{t_0}^t\big(  e^{(\lambda_{N+1} - \lambda_1)t} -  e^{(\lambda_{N+1} - \lambda_1)s} \big)e^{(\lambda_1-K_1)s}|\sigma(s)|ds.
$$
Thus we have
\begin{align*}
 e^{(\lambda_{N+1}-K_1)t}|\sigma(t)| &\le e^{(\lambda_{N+1}- K_1)t_0}|\sigma(t_0)| + \frac{K_1}{\lambda_{N+1} - \lambda_1}e^{(\lambda_1 - K_1)t_0} |\rho(t_0)|e^{(\lambda_{N+1}-\lambda_1)t} \\
&\hspace{3mm}+\frac{K_1^2}{\lambda_{N+1} - \lambda_1}e^{(\lambda_{N+1} - \lambda_1)t}\int_{t_0}^te^{(\lambda_1 - K_1)s}|\sigma(s)|ds\\
&\hspace{3mm} - \frac{K_1^2}{\lambda_{N+1} - \lambda_1}\int_{t_0}^te^{(\lambda_{N+1}-K_1)s}|\sigma(s)|ds.
\end{align*}
Put
$$
\phi(t) = \int_{t_0}^t e^{(\lambda_1 - K_1)s}|\sigma(s)|ds.
$$
Then
 $$
\phi'(t) = e^{(\lambda_1 - K_1)t}|\sigma(t)|,~\mbox{and}
$$
\begin{align}
e^{(\lambda_{N+1} - \lambda_1)t}\phi'(t) &\le e^{(\lambda_{N+1}-K_1)t_0}|\sigma(t_0)| + \frac{K_1}{\lambda_{N+1} - \lambda_1}e^{(\lambda_1 - K_1)t_0}|\rho(t_0)|e^{(\lambda_{N+1}-\lambda_1)t} \label{*} \\
&\hspace{3mm}+\frac{K_1^2}{\lambda_{N+1} - \lambda_1}e^{(\lambda_{N+1} - \lambda_1)t}\phi(t) -\frac{K_1^2}{\lambda_{N+1} - \lambda_1}\int_{t_0}^t e^{(\lambda_{N+1} - \lambda_1)s}\phi'(s)ds. \notag
\end{align}
The last two terms of \eqref{*} are reudced to
$$
K_1^2 \int_{t_0}^t e^{(\lambda_{N+1}-\lambda_1)s}\phi(s)ds.
$$

We now introduce another auxiliary function
$$
\psi(t) = \int_{t_0}^t e^{(\lambda_{N+1}-\lambda_1)s}\phi(s)ds.
$$
Then we have
\begin{align*}
\psi'(t) &= e^{(\lambda_{N+1} - \lambda_1)t}\phi(t), \\
\psi '' (t) &= (\lambda_{N+1} - \lambda_1) \psi '(t) + e^{(\lambda_{N+1} - \lambda_1)t} \phi '(t)
\end{align*}
Consequently the inequality \eqref{*} is reduced to
\begin{align}
\psi ''(t) - (\lambda_{N+1} &- \lambda_1)\psi '(t) - K_1^2\psi(t) \label{3.6} \\
&\le e^{(\lambda_{N+1} - K_1)t_0}|\sigma(t_0)| +\frac{K_1^2}{\lambda_{N+1} - \lambda_1}e^{(\lambda_1 - K_1)t_0}|\rho(t_0)|e^{(\lambda_{N+1} - \lambda_1)t}. \notag
\end{align}
With the choice of
\begin{align*}
\alpha &= \frac{-(\lambda_{N+1} - \lambda_1) + \sqrt{(\lambda_{N+1}-\lambda_1)^2+4K_1^2}}{2}~\mbox{and} \\
\beta &= \frac{(\lambda_{N+1} - \lambda_1) + \sqrt{(\lambda_{N+1}-\lambda_1)^2+4K_1^2}}{2},
\end{align*}
the left hand side of \eqref{3.6} becomes
$$
\psi ''(t) - \alpha \psi ' (t)+ \beta(\psi '(t) - \alpha \psi(t)).
$$
Multiply an integrating factor $e^{\beta t}$ and integrate to obtain
\begin{align*}
\psi ' - \alpha \psi &\le \frac{1}{\beta} e^{(\lambda_{N+1} - K_1)t_0} |\sigma(t_0)| \\
&\hspace{3mm}+ \frac{K_1^2}{(\lambda_{N+1} - \lambda_1 ) (\lambda_{N+1} -\lambda_1 + \beta) }e^{(\lambda_1 - K_1)t_0}|\rho(t_0)|e^{(\lambda_{N+1} - \lambda_1)t}.
\end{align*}
Integrating again, we obtain
\begin{align*}
\psi &\le \frac{1}{\alpha \beta}e^{(\lambda_{N+1} - K_1)t_0}|\sigma(t_0)|e^{\alpha( t - t_0)}  \\
&\hspace{3mm}+ \frac{K_1^2}{(\lambda_{N+1} - \lambda_1)(\lambda_{N+1} - \lambda_1 + \beta)(\lambda_{N+1} - \lambda_1 - \alpha)}e^{(\lambda_1 - K_1)t_0}|\rho(t_0)|e^{(\lambda_{N+1} - \lambda_1)t}.
\end{align*}
From the inequality \eqref{3.4}, we get
\begin{align*}
|\sigma(t)| &\le e^{-(\lambda_{N+1}-K_1)(t -t_0)}|\sigma(t_0)| + \frac{K_1}{\lambda_{N+1} - \lambda_1} e^{(\lambda_1 - K_1)t_0}|\rho(t_0)|e^{(K_1 - \lambda_1)t} \\
&\hspace{3mm}+\frac{K_1^2}{\alpha \beta}e^{(\lambda_{N+1} - K_1  - \alpha)t_0}|\sigma(t_0)|e^{-(\lambda_{N+1} - K_1 - \alpha)t}   \\
&\hspace{3mm}+  \frac{K_1^2}{(\lambda_{N+1} - \lambda_1)(\lambda_{N+1} - \lambda_1 + \beta)(\lambda_{N+1} - \lambda_1 - \alpha)}e^{(\lambda_1 - K_1)t_0}|\rho(t_0)|e^{(K_1 - \lambda_1)t}.
\end{align*}
In short notation, we have
\begin{equation} \label{+}
|\sigma(t)| \le K_2|\sigma(t_0)|e^{-(\lambda_{N+1} - K_1 - \alpha)(t-t_0)} + K_3 |\rho(t_0)| e^{(K_1 - \lambda_1)(t - t_0)}
\end{equation}
for some positive constants $K_2$, $K_3$ depending on $K_1$, $\lambda_1$ and $\lambda_{N+1}$. Plugging \eqref{+} into the inequality \eqref{3.3}, we obtain
$$
|\rho(t)| \le |\rho(t_0)|(1+K_4(t-t_0))e^{(K_1 - \lambda_1)(t - t_0)} + K_5 |\sigma(t_0)|e^{(K_1 - \lambda_1)(t -t_0)}.
$$
\end{proof}
As an immediate consequence of Lemma \ref{lemma3.1}, we observe that if $\rho(t_0) = 0$, then $\sigma(t)$ decays to zero exponentially.

\begin{lemma} \label{3.2}
$\{ Q\mathcal M_n \}_{n \in \mathbb N}$ is a Cauchy sequence with the Hausdorff distance $d_H$.
\end{lemma}

\begin{proof}
For any $m,n \in \mathbb N$ with $m \le n$, let $q_m \in Q \mathcal M_m$ and $q_n \in Q \mathcal M_n$. Then by definition, there exist solutions $u_m(t)$ and $u_n(t)$ of \eqref{1.1} such that
$$
Qu_m(0) = Qu_n(0)= 0,~Qu_m(m) = q_m,~\mbox{and}~Qu_n(n) = q_n.
$$
Let $v(t)$ be the solution of \eqref{1.1} for $ t \in [n-m, n]$ with $v(n-m) = Pu_n(m)$. Then from Lemma \ref{lemma3.1}, we find
\begin{align*}
|\sigma(n)| = | Qu_n(n) - Qv(n) | = |q_n - Qv(n)| \le K_2 |\sigma(m)|e^{-(\lambda_{N+1} - K_1 - \alpha)m},
\end{align*}
and thus
$$
\mbox{dist}(q_n, Q\mathcal M_m) \le \frac{K_0 K_2}{\lambda_{N+1}}e^{-(\lambda_{N+1}-K_1 - \alpha)m}.
$$

On the other hand, from Lemma \ref{lemma2.1}, there exists a solution $u(t)$ of \eqref{1.1} for $t \in [m-n, m]$ such that
$$
Qu(m-n) = 0~\mbox{and}~Pu(0) = Pu_m(0).
$$
By Lemma \ref{lemma3.1}, we get
$$
|\sigma(m)| = |Qu(m) - Qu_m(m)| = |Qu(m) - q_m| \le K_2|\sigma(0)|e^{-(\lambda_{N+1}-K_1 - \alpha)m}.
$$
Hence
$$
\mbox{dist}(Q\mathcal M_n, q_m) \le \frac{K_0K_2}{\lambda_{N+1}}e^{-(\lambda_{N+1} - K_1 - \alpha)m}.
$$
Consequently we obtain
$$
d_H(Q\mathcal  M_n, Q\mathcal M_m) \le \frac{K_0K_2}{\lambda_{N+1}}e^{-(\lambda_{N+1} - K_1 - \alpha)m},
$$
which completes the proof.
\end{proof}
We recall that the space of closed and bounded subsets of the Banach space $QH$ with the Hausdorff distance is complete. Hence the sequence $\{ QM_n \}$ converges to a closed bounded set, say $D_{\infty}$. For any $q_{\infty} \in D_{\infty}$, choose a sequence $\{q_n \in QM_n \}$ satisfying $q_n \to q_\infty$.
Then for each $n \in \mathbb N$, there exists $p_n \in PH$  such that $(p_n, q_n) \in \mathcal M_n$.  Moreover we observe that the dynamics of \eqref{1.1} are trivial outside of the absorbing ball by the truncation, and thus we may assume
$$
p_n \in B_R := \{ p \in PH : |p| \le R \}.
$$
In fact, for any $(p, q ) \in M_n$,
 $$q=0 ~\mbox{if}~ |p|>R.  \label{3.8}$$

Since $B_R$ is compact in $PH$, $\{ p_n \}_{n =1}^\infty$ has a convergent subsequence, say $p_{n_k} \to p_\infty$. We finally define
\begin{align*}
M_\infty = \{& (p_\infty, q_\infty) : \mbox{there~exists~a~sequence}~\{ (p_{n_k}, q_{n_k}) \}~\mbox{in}~\mathcal M_{n_k} \\
&\mbox{such~that}~(p_{n_k}, q_{n_k}) \to (p_\infty, q_\infty)~\mbox{as~}k \to \infty \}.
\end{align*}

\section{Main Results}

We recall that $\mathcal M_{-\infty}$ is constructed as the collection of backward bounded solutions whereas $M_\infty$ is a forward accumulation of solutions. Surprisingly, we first have;

\begin{theorem} \label{theorem4.1}
$\mathcal M_{\infty} \subset \mathcal M_{-\infty}$.
\end{theorem}

\begin{proof}

For any $(p_\infty, q_\infty) \in \mathcal M_\infty$, choose $(p_{n_k}, q_{n_k}) \in \mathcal M_{n_k}$ such that $p_{n_k} \to p_\infty$ and $q_{n_k} \to q_\infty$. For each $k \in \mathbb N$, there exists a solution $u_{n_k}(t) = p_{n_k}(t) + q_{n_k}(t)$ of \eqref{2.1} satisfying
$$
p_{n_k}(0) = p_{n_k},~q_{n_k}(0) = q_{n_k}~\mbox{and}~q_{n_k}(-n_k) = 0.
$$
By Lemma \ref{lemma2.2}, $\{ u_{n_k}(t) \}$ has a subsequence which converges to a backward bounded solution $u(t)$ with $Pu(0) = p_\infty$ and $Qu(0) = q_\infty$, which completes the proof.
\end{proof}

Now we are in a position to state a characterization of backward bounded solutions.

\begin{theorem} \label{theorem4.2}
The collection of backward bounded solutions of \eqref{1.1} is the graph of an upper hemicontinuous set-valued function from $PH$ to $QH$, which is invariant and contains the global attractor.
\end{theorem}

\begin{proof}
Consider a set-valued function $\Phi:PH \to QH$ defined by
\begin{align*}
\Phi(p_0) =\{&q_0 \in QH: (p(t),q(t))~ \mbox{is~ a~ backward ~bounded~solution~ of}~ \eqref{1.1}  \\
&~\mbox{such~that}~ p(0) = p_0, ~q(0) = q_0, ~\mbox{and}~q(t)~\mbox{is~bounded~on}~(-\infty,0] \}. \notag
\end{align*}
Then we see that the collection $\mathcal M_{-\infty}$ of backward bounded solutions of \eqref{1.1}  is the graph of $\Phi$, i.e.,
$$
\mathcal M_{-\infty} = \{ (p_0, q_0) : p_0 \in PH, q_0 \in \Phi(p_0)\}.
$$
It is clear that $\mathcal M_{-\infty}$ is invariant and contains the global attractor.

Furthermore, from (3.8), $\Phi$ has a compact support in $B_R$. The estimate (2.6) holds for any backward bounded solution and since the space $D(A^{1/2})$
is compactly imbedded in $H$,  the image of $\Phi$ is contained in a compact subset of $QH$.

We now show that the graph of $\Phi$ is closed in $H$.
For any $(p_\infty, q_\infty) \in {\overline{\mathcal M}}_\infty$, we choose a sequence $\{ (p_{\infty,n}, q_{\infty, n}) \}$ in $\mathcal M_\infty$ such that
$$
(p_{\infty,n}, q_{\infty, n})  \to  (p_{\infty}, q_{\infty})~\mbox{as}~n \to \infty.
$$
By the definition of $\mathcal M_\infty$, there is a sequence $\{ (p_{m_{k,n}}, q_{m_{k,n}}) \}$ in $\mathcal M_{m_{k,n}}$ such that
$$
(p_{m_{k,n}}, q_{m_{k,n}}) \to  (p_{\infty,n}, q_{\infty, n})~\mbox{as}~k \to \infty.
$$
We note here that $m_{k,l} \le m_{k,{l'}}$ if $l \le l'$, and for each $l$, $m_{k,l} \to \infty$ as $k \to \infty$.
By the diagonal argument, we can derive that
$$
(p_{m_{n,n}}, q_{m_{n,n}}) \to  (p_{\infty}, q_{\infty})~\mbox{as}~n \to \infty,
$$
and so $(p_\infty, q_\infty) \in \mathcal M_\infty$.

Similarly we can show that for any $p_0 \in PH$, $\Phi(p_0)$ is closed in $QH$. If we apply the closed graph theorem (e.g., Prop. 1.4.8 in \cite{AF}), we may conclude that $\Phi$ is upper hemicontinuous. This completes the proof.
\end{proof}

We finally observe that
\begin{theorem} \label{theorem4.3}
If $\Phi$ is a single valued function, then it holds
  $$\mathcal M_{\infty} = \mathcal M_{-\infty}. $$
\end{theorem}

\begin{proof}
Let $(p_{\infty}, q_{\infty}) \in \mathcal{M}_{-\infty}$. As argued in Lemma 2.2, we may take a sequence
$\{ u_{n_k} (t)=p_{n_k} (t) + q_{n_k} (t) \}$ of solutions of (2.1) for $t \in [-n_k, 0]$ such that 
$$p_{n_k} (0)= p_{\infty} ~ \mbox{and} ~~ q_{n_k} (-n_k )=0,$$
and $u_{n_k} (t)$ converges uniformly to a backward bounded solutuon as $k \to \infty$.
Now we have 
$$(p_{n_k} (0), q_{n_k} (0) ) \in \mathcal M_{n_k}, ~\mbox{and}~ (p_{n_k} (0), q_{n_k} (0) ) \to (p_{\infty}, q_{\infty}')  ~\mbox{as} ~ k \to \infty.$$
 From the assumption, we may conclude that $q_{\infty} = q_{\infty}'$ and  so $(p_{\infty}, q_{\infty}) \in \mathcal{M}_{\infty}$.

\end{proof}

\end{document}